\documentclass[a4paper,12pt]{amsart}
\usepackage[english]{babel}
\usepackage{amsmath,amsthm,amssymb,amsfonts}
\usepackage{multicol}
 
\usepackage[pdftex]{color}
\usepackage[bookmarks=true,hyperindex,pdftex,colorlinks, citecolor=blue,linkcolor=blue, urlcolor=blue]{hyperref}
\usepackage[dvipsnames]{xcolor}
\usepackage{mathtools}
\usepackage{graphicx}
\usepackage[titletoc]{appendix}
\usepackage{tikz}

\usepackage[normalem]{ulem}

\usetikzlibrary{matrix}

\newtheorem{theorem}{Theorem}[section]
\newtheorem{lemma}[theorem]{Lemma}

\newtheorem{proposition}[theorem]{Proposition}
\newtheorem{corollary}[theorem]{Corollary}

\theoremstyle{definition}
\newtheorem{definition}[theorem]{Definition}

\theoremstyle{remark}
\newtheorem{remark}[theorem]{Remark}
\numberwithin{equation}{section}

\newcommand{\R}{\mathbb{R}}

\newcommand{\N}{\mathbb{N}}
\newcommand{\K}{\mathbb{K}}

\DeclareMathOperator{\co}{co}
\DeclareMathOperator{\aco}{aco}

\DeclareMathOperator{\sign}{sign}

\newcommand{\nn}[1]{{\left\vert\kern-0.25ex\left\vert\kern-0.25ex\left\vert #1 
		\right\vert\kern-0.25ex\right\vert\kern-0.25ex\right\vert}}

\renewcommand{\geq}{\geqslant}
\renewcommand{\leq}{\leqslant}

\newcommand{\norm}[1]{\left\Vert#1\right\Vert}

\newcommand{\NA}{\operatorname{NA}}
\newcommand{\spann}{\operatorname{span}}

\newcommand{\ext}[1]{\operatorname{ext}\left(#1\right)}

\newcommand{\pten}{\ensuremath{\widehat{\otimes}_\pi}}
\newcommand{\iten}{\ensuremath{\widehat{\otimes}_\varepsilon}}
\newcommand{\eps}{\varepsilon}
\newcommand{\sten}{\ensuremath{\widehat{\otimes}_{\pi,s,N}}}
\newcommand{\siten}{\ensuremath{\widehat{\otimes}_{\eps,s,N}}}
\newcommand{\cconv}{\overline{\co}}

\begin{document}

\title{On Norm-Attainment in (Symmetric) Tensor Products}

\author[Dantas]{Sheldon Dantas}
\address[Dantas]{Departament de Matem\`{a}tiques and Institut Universitari de Matem\`{a}tiques i Aplicacions de Castell\'o (IMAC), Universitat Jaume I, Campus del Riu Sec. s/n, 12071 Castell\'o, Spain. \newline
	\href{http://orcid.org/0000-0001-8117-3760}{ORCID: \texttt{0000-0001-8117-3760} } }
\email{\texttt{dantas@uji.es}}

\author[Garc\'ia-Lirola]{Luis C. Garc\'ia-Lirola}
\address[Garc\'ia-Lirola]{Departamento de Matem\'aticas, Universidad de Zaragoza, 50009, Zaragoza, Spain \newline
	\href{https://orcid.org/0000-0002-1966-1330}{ORCID: \texttt{0000-0002-1966-1330} }}
\email{\texttt{luiscarlos@unizar.es}}

\author[Jung]{Mingu Jung}
\address[Jung]{Basic Science Research Institute and Department of Mathematics, POSTECH, Pohang 790-784, Republic of Korea \newline
	\href{http://orcid.org/0000-0000-0000-0000}{ORCID: \texttt{0000-0003-2240-2855} }}
\email{\texttt{jmingoo@postech.ac.kr}}

\author[Rueda Zoca]{Abraham Rueda Zoca}
\address[Rueda Zoca]{Universidad de Murcia, Departamento de Matem\'aticas, Campus de Espinardo 30100 Murcia, Spain
	\newline
	\href{https://orcid.org/0000-0003-0718-1353}{ORCID: \texttt{0000-0003-0718-1353} }}
\email{\texttt{abraham.rueda@um.es}}
\urladdr{\url{https://arzenglish.wordpress.com}}

\begin{abstract} In this paper, we introduce a concept of norm-attainment in the projective symmetric tensor product $\sten X$ of a Banach space $X$, which turns out to be naturally related to the classical norm-attainment of $N$-homogeneous polynomials on $X$. Due to this relation, we can prove that there exist symmetric tensors that do not attain their norms, which allows us to study the problem of when the set of norm-attaining elements in $\sten X$ is dense. We show that the set of all norm-attaining symmetric tensors is dense in $\sten X$ for a large set of Banach spaces as $L_p$-spaces, isometric $L_1$-predual spaces or Banach spaces with monotone Schauder basis, among others. Next, we prove that if $X^*$ satisfies the Radon-Nikod\'ym and the approximation property, then the set of all norm-attaining symmetric tensors in $\sten X^*$ is dense. From these techniques, we can present new examples of Banach spaces $X$ and $Y$ such that the set of all norm-attaining tensors in the projective tensor product $X \pten Y$ is dense, answering positively an open question from the paper \cite{DJRR}.
\end{abstract}

\thanks{ }


\maketitle

\thispagestyle{plain}

\section{Introduction}

Recently, it was studied the set of the nuclear operators $T\colon X \longrightarrow Y$ between two Banach spaces $X, Y$ that attains their nuclear norm in the sense that 
\[ T=\sum_{n=1}^\infty x_n^*\otimes y_n, \text{ and } \norm{T}_{\mathcal N}=\sum_{n=1}^\infty \norm{x_n^*}\norm{y_n}\]
for some $(x_n^*)_{n=1}^\infty\subseteq X^*$ and $(y_n)_{n=1}^\infty\subseteq Y$ \cite{DJRR}. 
 From a practical point of view, it has been shown that this new concept has great connections with different norm-attainment concepts like the norm-attainment of bounded functionals, the norm-attainment of operators and the norm-attainment of bilinear forms coming from the identification  $(X\pten Y)^*=\mathcal L(X,Y^*)=\mathcal B(X\times Y)$ (see \cite[Proposition~3.10 and Corollary~3.11]{DJRR}). In this paper, we focus on a related norm-attaining notion on the $N$-fold symmetric tensor product of a Banach space $X$: we say that an element $z\in \sten X$ attains its norm if 
 \[ z =\sum_{n=1}^\infty \lambda_n x_n^N, \text{ and } \norm{z}=\sum_{n=1}^\infty |\lambda_n|\]
 for some $(\lambda_n)_{n=1}^\infty \subseteq \mathbb K$ and $(x_n)_{n=1}^\infty \subseteq X$. 
 Due to the fact that the dual $(\sten X)^*$ of the $N$-fold symmetric tensor product of a Banach space $X$ is identified with the space $\mathcal{P}(^N X)$ of all $N$-homogeneous polynomials on $X$, this norm-attainment notion turns out to be closely related to the theory of norm-attaining homogeneous polynomials for which the reader is referred to \cite{AAGM, CLM, CK96, KimLee}.

We present a characterization for the set of all norm-attaining elements in $\sten X$, denoted by $\NA_{\pi, s, N}(X)$, and we use it to prove that if every element in $\sten X$ attains its norm, then the set $\NA(^N X)$ of all $N$-homogeneous polynomials which attain their norms in dense in $\mathcal{P}(^N X)$. This, together with the fact that there are Banach spaces $X$ such that the set $\NA(^N X)$ is not dense in $\mathcal{P}(^N X)$ (see \cite{AAGM, JP}), allows us to get our first examples of spaces $X$ so that we can guarantee the existence of non-norm-attaining elements in $\sten X$. 
As there exist elements in $\sten X$ which do not attain their norms, it is natural to ask when the set of norm-attaining elements in $\sten X$ forms a dense subset. We prove that, under the metric $\pi$-property (see \cite{CASAZZA, JRZ}) on $X$, the denseness of $\NA_{\pi,s,N}(X)$ holds from the fact that every tensor in $\sten Z$ attains its norm whenever $Z$ is a finite dimensional space. This shows that, for a large class of Banach spaces, as for instance $L_p$-spaces, $L_1$-predual spaces, and Banach spaces with monotone Schauder basis, the set $\NA_{\pi,s,N}(X)$ is dense in $\sten X$. 
We also present a result not covered by the previous ones which holds under the Radon-Nikod\'ym property assumption. More precisely, we show that if $X^*$ has the Radon-Nikod\'ym property and the approximation property, then the set of tensors in $\sten X^*$ which attain their norms is dense. Moreover, we observe that the problem whether the set $\NA_{\pi, s, N}(X)$ is dense in $\sten X$ for every Banach space, is separably determined. 

We finish the paper by considering the set $\NA_{\pi}(X \pten Y)$ of all norm-attaining tensors in $X \pten Y$ and obtaining some positive results on the denseness of the set $\NA_{\pi} (X \pten Y)$. For instance, we prove that if $X$ is the convex hull of a finite set and $Y$ is a dual space, then {\it every} element in $X \pten Y$ attains its norm, which seems to be surprising somehow since there exists a Banach space $X$ so that $\NA_{\pi} (X \pten \ell_2^2) \not= X \pten \ell_2^2$ (see \cite[Example~3.12.(a)]{DJRR}). This result allows us to show that if $X$ is a polyhedral Banach space with the metric $\pi$-property, then the set $\NA_{\pi}(X \pten Y)$ is dense in $X \pten Y$ whenever $Y$ is a dual space. 
Moreover, in the same line as the symmetric tensor product case, we give a positive answer to an open question from \cite{DJRR} by proving that $\NA (X^* \pten Y^*)$ is dense in $X^* \pten Y^*$, provided that $X^*$ and $Y^*$ both have the Radon-Nikod\'ym property, and at least one of them has the approximation property.


\section{Notation and Preliminary Results}

In this section, we give the necessary notation and some preliminaries results we will be using throughout the paper.

The letters $X, Y$, and $Z$ stand for Banach spaces over the field $\mathbb{K}$ which will be $\mathbb{R}$ or $\mathbb{C}$. We denote by $B_X$ and $S_X$ the closed unit ball and unit sphere of $X$, respectively. Given a subset $B \subseteq X$, we denote the convex hull of $B$ by $\co(B)$. The symbol $\aco(B)$ stands for the absolutely convex (i.e., the convex and balanced) hull of the set $B$. If $A, B$ are subsets of $X, Y$, respectively, we denote by $A \otimes B$ the set $\{x \otimes y\in X\otimes Y: x \in A, y \in B\}$. Given a subset $C$ of $X$, a point $x \in C$ is said to be an extreme point of $C$ if $x$ cannot be written as a convex combination of points in $C$ which are different from $x$ itself. We denote by $\ext{C}$ the set of all extreme points of $C$.

For two Banach spaces $X$ and $Y$, the symbol $\mathcal{L}(X, Y)$ stands for the space of all bounded linear operators from $X$ into $Y$. By $\mathcal{B}(X \times Y)$ we mean the space of all bilinear forms on $X \times Y$ taking values in $\K$. The Banach space of all scalar-valued $N$-homogeneous polynomials on $X$ is denoted by $\mathcal{P}(^N X)$, which is endowed with the norm $\|P\| = \sup_{x \in B_X} |P(x)|$ for every $P \in \mathcal{P}(^N X)$. In this case, $P$ is said to attain its norm when this supremum becomes a maximum. We denote by $\NA(^N X)$ the set of all $N$-homogeneous polynomials which attain their norms on $X$. For background on homogeneous polynomials we refer the refer to \cite{Din, HJ, Muj}.

The projective and injective tensor product between $X$ and $Y$, denoted by $X \pten Y$ and $X\iten Y$, are the completion of the algebraic tensor product $X \otimes Y$ endowed with the norms 
\begin{equation} \label{projective-norm}
\|z\|_{\pi} := \inf \left\{ \sum_{i=1}^n \|x_i\| \|y_i\|: z = \sum_{i=1}^n x_i \otimes y_i \right\},
\end{equation}
where the infimum is taken over all such representations of $z$, 
and 
\begin{equation*}
	\norm{\sum_{i=1}^n x_i\otimes y_i}_\varepsilon = \sup \left\{ \left| \sum_{i=1}^n x^*(x_i)y^*(y_i) \right|: x^*\in B_{X^*}, y^*\in B_{Y^*}\right\}.
\end{equation*}

 It is well known that $B_{X \pten Y} = \overline{\co}{(B_X \otimes B_Y)}$ and that $(X \pten Y)^* = \mathcal{B}(X \times Y) = \mathcal{L}(X, Y^*)$. There is a canonical operator $J\colon X^* \pten Y \longrightarrow \mathcal{L}(X, Y)$ with $\|J\| = 1$ defined by $z = \sum_{n=1}^{\infty} x_n^* \otimes y_n \mapsto L_z$, where $L_z: X \longrightarrow Y$ is given by $L_z(x) = \sum_{n=1}^{\infty} x_n^*(x) y_n$. The operators that arise in this way are called nuclear operators and we denote them by $\mathcal{N}(X, Y)$ endowed with the nuclear norm
\begin{equation} \label{nuclear-norm}
\|T\|_{\mathcal{N}} := \inf \left\{ \sum_{n=1}^{\infty} \|x_n^*\| \|y_n\|: T(x) = \sum_{n=1}^{\infty} x_n^*(x) y_n \right\}
\end{equation}
where the infimum is taken over all representations of $T$ of that form. Let us notice that every nuclear operator is the limit (in the operator norm) of a sequence of finite-rank operators, so every nuclear operator is compact.

Recall that a Banach space $X$ satisfies the approximation property (AP, for short) if for every compact subset $K$ of $X$ and for every $\eps > 0$, there exists a finite-rank operator $T\colon X \longrightarrow X$ such that $\|T(x) - x\| \leq \eps$ for every $x \in K$. It turns out that whenever $X^*$ or $Y$ has the approximation property, then $X^* \pten Y = \mathcal{N}(X, Y)$. For a detailed account on tensor products and nuclear operators, we refer the reader to \cite{DF, rya}.

Given a natural number $N \in \N$, we denote by $z^N$ the element $z \otimes \stackrel{N}{\ldots} \otimes z\in X\otimes \stackrel{N}{\ldots} \otimes X$ for every $z \in X$. The ($N$-fold) projective symmetric tensor product of $X$, denoted by $\sten X$, is the completion of the linear space $\otimes_{\pi, s, N} X$, generated by $\{z^N : z \in X\}$, under the norm given by
\begin{equation} \label{stn}
\|z\|_{\pi,s,N} := \inf \left\{ \sum_{k=1}^n |\lambda_k| : z:= \sum_{k=1}^n \lambda_k x_k^N, n \in \N, x_k \in S_X, \lambda_k \in \K \ \right\},
\end{equation}
where the infimum is taken over all the possible representations of $z$. Its topological dual $\left( \sten X \right)^*$ can be identified (there exists an isometric isomorphism) with $\mathcal{P}(^N X)$. Indeed, every polynomial $P \in \mathcal{P}(^N X)$ acts as a linear functional on $\sten X$ through its associated symmetric $N$-linear form $\overline{P}$ and satisfies 
\begin{equation*} 
P(x) = \overline{P}(x, \ldots, x) = \langle P, x^N \rangle  
\end{equation*} 
for every $x \in X$. We also have that $B_{\sten X} = \overline{\aco}(\{x^N: x \in S_X\})$. To save notation, by a symmetric tensor we will refer to a generic element of $\sten X$. For more information about symmetric tensor products, we send the reader to \cite{F1} and also to recent papers as \cite{BR, CD, CG}.


Throughout the paper, we will be interested in studying the concepts of norm-attainment on $X \pten Y$, $\mathcal{N}(X, Y)$, and $\sten X$ meaning that their norms (\ref{projective-norm}), (\ref{nuclear-norm}), and (\ref{stn}) are respectively attained. More precisely, we have the following definitions, which will be our main notions in this paper:
\begin{itemize}
	
	\item[(1)] $z \in X \pten Y$ {\it attains its projective norm} if there is a bounded sequence $(x_n, y_n) \subseteq X \times Y$ with $\sum_{n=1}^{\infty} \|x_n\| \|y_n\| < \infty$ such that $z=\sum_{n=1}^\infty x_n\otimes y_n$ and that $\|z\|_{\pi} = \sum_{n=1}^{\infty} \|x_n\| \|y_n\|$. In this case, we say that $z$ is a {\it norm-attaining tensor}.
	\vspace{0.3cm} 
	
	\item[(2)] $T \in \mathcal{N}(X, Y)$ {\it attains its nuclear norm} if there is a bounded sequence $(x_n^*, y_n) \subseteq X^* \times Y$ with $\sum_{n=1}^{\infty} \|x_n^*\|\|y_n\| < \infty$ such that $T=\sum_{n=1}^\infty x_n^*\otimes y_n$ and that $\|T\|_{\mathcal{N}} = \sum_{n=1}^{\infty} \|x_n^*\| \|y_n\|$. In this case, we say that $T$ is a {\it norm-attaining nuclear operator}.
	
	\vspace{0.3cm} 
	
	\item[(3)] $z \in \sten X$ {\it attains its projective symmetric norm} if there are bounded sequences $(\lambda_n)_{n=1}^{\infty} \subset \K$ and $(x_n)_{n=1}^{\infty} \subseteq B_X$ such that $\|z\|_{\pi, s, N} = \sum_{n=1}^{\infty} |\lambda_n|$ for $z = \sum_{n=1}^{\infty} \lambda_n x_n^N$. In this case, we say that $z$ is a {\it norm-attaining symmetric tensor}.
\end{itemize}

\noindent
When there is no confusion of misunderstanding and it is clear on what spaces we are working with, we denote the norms $\| \cdot \|_{\pi}$, $\| \cdot \|_{\mathcal{N}}$, and $\|\cdot\|_{\pi,s,N}$ simply by $\|\cdot\|$. Therefore, we set
\begin{itemize}
\item[(i)] $\NA_{\pi} (X \pten Y) = \Big\{ z \in X \pten Y: z \ \mbox{attains its projective norm} \Big\}$,
\vspace{0.1cm}
\item[(ii)] $\NA_{\mathcal{N}} (X,Y) = \Big\{ T \in \mathcal{N}(X, Y): T \ \mbox{attains its nuclear norm} \Big\}$,
\vspace{0.1cm}
\item[(iii)] $\NA_{\pi,s,N}(X) = \Big\{ z \in \sten X: z \ \mbox{attains its symmetric norm} \Big\}$.	
\end{itemize}



Recall that a subspace $Y$ of a Banach space $X$ is said to be an ideal of $X$ if for every finite-dimensional subspace $E$ of $X$ and every $\eps > 0$, there is a linear operator $T \in \mathcal{L} (E , Y)$ such that $T(e) = e$ for every $e \in E \cap Y$ and $\|T\| \leq 1 + \eps$. 
Let us notice that $1$-complemented subspaces are ideals and that the concept of being an ideal of $X$ coincides with the one of locally complemented subspace of $X$ (see \cite{Kalton}). The following result is motivated by \cite[Theorem~1.(a)]{Rao}, where the author proves that if $X$ and $Z$ are Banach spaces and $Y$ is an ideal of $Z$, then $X \pten Y$ is a subspace of $X \pten Z$ and it is an ideal. In what follows, $\sten Y$ being an isometric subspace means that if we consider the natural embedding of it into $\sten X$, then the norms in $\sten Y$ and $\sten X$ coincide on $\sten Y$.

\begin{theorem} \label{ideal} Let $X$ be a Banach space and $Y$ an ideal of $X$. Then, $\sten Y$ is an isometric subspace of $\sten X$. 
\end{theorem}
\begin{proof} Notice first that, by a denseness argument, it is enough to prove the theorem for $z = \sum_{i=1}^n \lambda_i y_i \in \otimes_{\pi, s, N} Y \subseteq \otimes_{\pi, s, N} X$. By the definition of the norm (see (\ref{stn})), we have that $\|z\|_{\sten X} \leq \|z\|_{\sten Y}$. Now, let us prove the other inequality.

Let $\eps > 0$ be given. Since the norm on a symmetric tensor product is finitely generated (see \cite[Subsection 2.2]{F1}), there exists a finite-dimensional subspace $F$ of $X$ containing $\{y_1, \ldots, y_n\}$ such that $\|z\|_{\sten F} < \|z\|_{\sten X} + \eps$. Since $Y$ is an ideal in $X$, there exists a linear operator $T \in \mathcal{L} ( F , Y)$ such that $\|T\| \leq \sqrt[N]{1 + \eps}$ and $T(y_i) = y_i$ for every $i=1,\ldots, n$. Let us define $T^N \in \mathcal{L} (\sten F , \sten Y)$ by $T^N(m^N) := T(m)^N$ for every $m \in F$. This operator is well-defined and satisfies $\|T^N\| = \|T\|^N \leq 1 + \eps$ (see \cite[Subsection 2.2]{F1}). Therefore, we have that 
\begin{equation*}
\sum_{i=1}^n \lambda_i y_i^N = \sum_{i=1}^n \lambda_i T(y_i)^N = \sum_{i=1}^n T^N (y_i^N) = T^N \left( \sum_{i=1}^n \lambda_i y_i^N \right)
\end{equation*} 
and then 
	\begin{eqnarray*}
		\left\| \sum_{i=1}^n \lambda_i y_i^N \right\|_{\sten Y} = \left\| T^N \left( \sum_{i=1}^n \lambda_i y_i^N \right) \right\|_{\sten Y} &\leq& \|T^N\| \left\| \sum_{i=1}^n \lambda_i y_i^N \right\|_{\sten F} \\
		&\leq& (1 + \eps)\|z\|_{\sten F} \\
		&<& (1 + \eps) (\|z\|_{\sten X} + \eps).	
	\end{eqnarray*}	
	Since $\eps > 0$ is arbitrary, $\|z\|_{\sten Y} \leq \|z\|_{\sten X}$ and we are done.	
\end{proof}

We will be using also the following straightforward fact.  

\begin{lemma} \label{estimation-symmetric} Let $X$ be a Banach space. Let $\eps > 0$ and $x, y \in S_X$ Then, 
	\begin{equation*}
	\|x^N - y^N\|_{\sten X} \leq \frac{N^{N+1}}{N!} \|x-y\| .
	\end{equation*}
\end{lemma}

\begin{proof} By the polarization constant (see \cite[Subsection 2.3]{F1}), we have that
	\begin{equation*}
	\|x^N - y^N\|_{\sten X} \leq \frac{N^N}{N!} \|x^N - y^N\|_{X \pten X \dots \pten X}.
	\end{equation*}
	Now, let us notice that
	\begin{equation*}
	x^N - y^N = \sum_{k=1}^N x^{N-k} \otimes (x - y) \otimes y^{k-1}.
	\end{equation*}	
	This proves the statement since
	\begin{equation*}
	\|x^N - y^N\|_{X \pten \dots \pten X} \leq \sum_{k=1}^N \|x\|^{N-k} \|x-y\| \|y\|^{k-1} = N \|x - y\|.
	\end{equation*}	
\end{proof}

\section{Results for Symmetric Tensor Products} 

In this section we show that there are symmetric tensors that {\it do not} attain their norms and study the denseness problem for norm-attaining elements in $\sten X$. We start by giving a relation between the concepts of norm-attainment for symmetric tensors and $N$-homogeneous polynomials.

\begin{theorem} \label{everyelementnormattaining} Let $X$ be a Banach space and suppose that every element in $\sten X$ attains its norm. Then, the set of all $N$-homogeneous polynomials that attain their norms is dense in the space of all $N$-homogeneous polynomials. In other words,
	\begin{equation*} 
	\overline{\NA(^N X)} = \mathcal{P}(^N X).
	\end{equation*}

\end{theorem}

In order to prove Theorem~\ref{everyelementnormattaining}, we present a characterization for elements of $\sten X$ to attain their norms. We have the following result, which is the counterpart of \cite[Theorem~3.1]{DJRR} for symmetric tensors and homogeneous polynomials. We denote by $\sign(\lambda)$ the complex number $\frac{\overline{\lambda}}{\vert \lambda\vert}$ for each $\lambda \in \mathbb{C} \setminus \{0\}$. 

\begin{lemma} \label{characterization} Let $X$ be a Banach space and let 
\begin{equation*} 	
z = \sum_{n=1}^{\infty} \lambda_n x_n^N \in \sten X
\end{equation*} 
where $\lambda_n \in \mathbb{C} \setminus \{0\}$ and $(x_n)_{n=1}^{\infty} \subseteq S_X$. Then, the following statements are equivalent.
	\begin{itemize}
		\item[(1)] $\|z\| = \sum_{n=1}^{\infty} |\lambda_n|$; in other words, $z \in \NA_{\pi, s, N} (X)$.
		\item[(2)] There exists $P \in S_{\mathcal{P}(^NX)}$ such that $P(x_n) = \sign(\lambda_n), \forall n \in \N$.
		\item[(3)] Every $P \in S_{\mathcal{P}(^NX)}$ such that $P(z) = \|z\|$ satisfies $P(x_n) = \sign(\lambda_n), \forall n \in \N$.
	\end{itemize}
\end{lemma}

\begin{proof} Let us suppose that (1) holds. Pick any $P \in \left( \sten X \right)^* = \mathcal{P}(^N X)$ with $\|P\| = 1$ and $P(z) = \|z\|$. We have that
	\begin{equation*}
	\sum_{n=1}^{\infty} |\lambda_n| = \|z\| = P(z) = \sum_{n=1}^{\infty} \lambda_n \text{Re} P(x_n) \leq \sum_{n=1}^{\infty} |\lambda_n|,
	\end{equation*}
	which implies that $P(x_n) = \sign(\lambda_n)$ for every $n \in \N$. This shows that (3) holds. The implication (3) $\Rightarrow (2)$ is immediate. Assume now that (2) holds. Then, there exists $P \in \mathcal{P}(^N X)$ with $\|P\| = 1$ such that $P(x_n) = \sign(\lambda_n)$ for every $n \in \N$. So,
	\begin{equation*}
	\sum_{n=1}^{\infty} |\lambda_n| \geq \|z\| \geq P(z) = \sum_{n=1}^{\infty} \lambda_n P(x_n) = \sum_{n=1}^{\infty} |\lambda_n|
	\end{equation*} 
	and this implies $\|z\| = \sum_{n=1}^{\infty} \lambda_n$. Therefore, (2) implies (1).
\end{proof}

By using Lemma~\ref{characterization} above, we can now prove Theorem~\ref{everyelementnormattaining}.

\begin{proof}[Proof of Theorem~\ref{everyelementnormattaining}] Let $\eps > 0$ and $P \in \mathcal{P}(^N X) = (\sten X)^*$ with $\|P\| = 1$ be given. By the Bishop-Phelps theorem for the Banach space $\sten X$, there are $P_0 \in \mathcal{P}(^N X)$ with $\|P_0\| = 1$ and $z_0 \in S_{\sten X}$ such that 
\begin{equation*} 	
P(z_0) = 1 \ \ \ \ \mbox{and} \ \ \ \ \|P_0 - P\| < \eps.
\end{equation*} 
By hypothesis we have that $z_0 \in \NA_{\pi, s, N} (X)$. So, there are $(\lambda_n)_{n=1}^{\infty} \subseteq \mathbb{R} \setminus \{0\}$ and $(x_n)_{n=1}^{\infty} \subseteq S_X$ such that $\|z_0\| = \sum_{n=1}^{\infty} \lambda_n $ for $z_0 = \sum_{n=1}^{\infty} \lambda_n x_n^N$. By Lemma~\ref{characterization}, $P_0(x_n) = \sign(\lambda_n)$ for every $n \in \N$. In particular, $P_0 \in \NA(^N X)$ and we are done.
\end{proof}

Now we are able to present some examples where there exist symmetric tensors $z$ which do not attain their norms.

\begin{remark} \label{there-are-non-norm-attaining} It is known (see \cite{AAGM, JP}) that if $X = d_{*}(w, 1)$ with $w \in \ell_2 \setminus \ell_1$, the predual of Lorentz sequence space, then the set $\mathcal{P}(^N X)$, for $N \geq 2$, of all norm-attaining $N$-homogeneous polynomials on $X$, is {\it not} dense in $\mathcal{P}(^N X)$. Thus, Theorem~\ref{everyelementnormattaining} implies that there exists an element $z$ in $\sten X$ which does not attain its norm. 
\end{remark} 	
	
In contrast to Remark~\ref{there-are-non-norm-attaining}, when $X$ is finite-dimensional, we do have that {\it every} symmetric tensor is norm-attaining (we send the reader also to Theorem~\ref{theorem-polyhedral-dual} for an analogous phenomenon on projective tensor products). Its proof can be obtained by arguing as in \cite[Proposition~3.5]{DJRR} with the aid of the fact that a convex hull of a compact set in a finite dimensional space is again compact and that $B_{\sten X} = \overline{\aco}(\{x^N: x \in S_X\})$.

\begin{proposition} \label{finite-dimensional} Let $X$ be a finite dimensional Banach space. Then, every symmetric tensor attains its projective symmetric tensor norm. In other words, 
\begin{equation*} 	
\NA_{\pi, s, N}( X) = \sten X.
\end{equation*}
\end{proposition}

As promised, we shall investigate when it is possible to approximate an arbitrary element $z \in \sten X$ by a norm-attaining symmetric tensor. Similarly to what it is done in \cite{DJRR}, this is achieved under the assumption that $X$ contains ``many'' 1-complemented subspaces. 

\begin{definition} Let $X$ be a Banach space. We say that $X$ has the metric $\pi$-property if given $\eps > 0$ and $\{x_1, \ldots, x_n\} \subseteq S_X$,  we can find a finite dimensional 1-complemented subspace $M \subseteq X$ and  $x_i' \in M$ with $\|x_i - x_i'\| < \eps$ for every $i=1,\ldots,n$. 
\end{definition}

We invite the reader to \cite{CASAZZA} (and also to \cite{JRZ, Lindenstrauss}) for more information about $\pi$-properties. Moreover, \cite[Example~4.12]{DJRR} sums up known examples of Banach spaces satisfying the metric $\pi$-property. Just to name a few, it is known that $L_p$-spaces, $L_1$-predual spaces, and Banach spaces with a finite dimensional decomposition with decomposition constant 1 satisfy such a property. 
Now, we present the following result analogous to \cite[Theorem~4.8]{DJRR}.

\begin{theorem} \label{propertyP} Let $X$ be a Banach space with the metric $\pi$-property. Then, every symmetric tensor can be approximated by symmetric tensors which attain their norms. In other words,
	\begin{equation*}
	\overline{\NA_{\pi, s, N} ( X)}^{\|\cdot\|_{\pi, s, N}} = \sten X.
	\end{equation*}
\end{theorem}

\begin{proof} Let $u \in S_{\sten X}$ and $\eps > 0$ be given. There are $(\lambda_n)_{n=1}^\infty \subseteq \mathbb{R} \setminus \{0\}$ and $(x_n)_{n=1}^{\infty} \subseteq S_X$ such that
	\begin{equation*}
	u = \sum_{n=1}^{\infty} \lambda_n x_n^N \ \ \ \mbox{and} \ \ \ \sum_{n=1}^{\infty} |\lambda_n| < 1 + \eps. 
	\end{equation*}
	Find $k \in \N$ large enough such that
	\begin{equation*}
	\|u - z\| < \frac{\eps}{2} \ \ \ \mbox{for} \ \ \ z = \sum_{n=1}^k \lambda_n x_n^N \in \otimes_{\pi, s, N} X.
	\end{equation*}
	Since $X$ has the metric $\pi$-property, we can find a finite dimensional space $M$ of $X$ which is $1$-complemented and such that, for every $n \in \{1,\ldots,k\}$, there exists $x_n' \in M$ such that
	\begin{equation*}
	\|x_n - x_n'\| < \frac{N!}{N^{N+1}} \cdot \frac{\eps}{4}.
	\end{equation*}	
	Define $z' := \sum_{n=1}^k \lambda_n (x_n')^N \in \otimes_{\pi, s, N} M$. Since $M$ is finite dimensional, by Proposition~\ref{finite-dimensional} we have $z' \in \NA_{\pi, s, N} (M)$, and, since $\norm{z'}_{\sten M}=\norm{z'}_{\sten X}$, we also have $z' \in \NA_{\pi, s, N} (X)$. Finally, by using Lemma~\ref{estimation-symmetric}, we have that 
	\begin{equation*}
	\|z' - z\| = \left\| \sum_{n=1}^k \lambda_n \Big( (x_n')^N - x_n^N \Big) \right\| \leq \sum_{n=1}^k |\lambda_n| \| (x_n')^N - x_n^N \| < \frac{\eps}{2}.
	\end{equation*}
Therefore, $\|z' - u\| < \eps$ and we are done.	
\end{proof}

Before proceeding, let us use Theorem~\ref{propertyP} to point out the following observation on the hypothesis of Theorem~\ref{everyelementnormattaining}.

\begin{remark} In Theorem~\ref{everyelementnormattaining}, the assumption that every element of $\sten X$ attains its norm {\it cannot} be relaxed to the case that $\NA_{\pi, s, N} (X)$ is dense in $\sten X$. Indeed, if $X = d_{*}(w, 1)$ with $w \in \ell_2 \setminus \ell_1$, then $X$ has monotone symmetric basis (see, for instance, \cite[Proposition~2.2]{W} and \cite[Lemma~2.2]{JP}) and, therefore, satisfies the metric $\pi$-property (see, for instance, \cite[Example~4.12]{DJRR}), which implies that $\NA_{\pi,s,N}(X)$ is dense in $\sten X$ by Theorem~\ref{propertyP}. On the other hand, as we already have mentioned in Remark~\ref{there-are-non-norm-attaining}, the set of all norm-attaining $N$-homogeneous polynomials is not dense in $\mathcal{P}(^N X)$ for $N \geq 2$.   
\end{remark} 

Our next goal will be obtaining the following result on the denseness of norm-attaining elements in $\sten X^*$ under the hypothesis of Radon-Nikod\'ym property (for short, RNP), see Theorem~\ref{th:RNPnuclear} and Corollary~\ref{cor:RNPtensor} for its counterpart for nuclear operators and projective tensor products, respectively.

\begin{theorem} \label{RNP+AP=symmetric} Let $X$ be a Banach space. Suppose that $X^*$ has the RNP and the AP. Then, every symmetric tensor in $\sten X^*$ can be approximated by symmetric tensors that attain their norms. In other words,
	\begin{equation*} 	
		\overline{\NA_{\pi, s, N} (X^*)}^{\|\cdot\|_{\pi,s,N}} = \sten X^*. 
	\end{equation*} 
\end{theorem}

In order to prove Theorem~\ref{RNP+AP=symmetric}, we need two preliminary results. Let us start with the following general lemma for spaces satisfying the RNP, which will also be used to prove Theorem~\ref{th:RNPnuclear}.

\begin{lemma}\label{lemma:RNP} Let $X$ be a Banach space with the RNP. Then,  
	\[A := \left\{ x=\sum_{i=1}^n \lambda_i x_i \in X: \lambda_1,\ldots, \lambda_n>0, x_1,\ldots, x_n\in \ext{B_X}, \norm{x}=\sum_{i=1}^n \lambda_i \right\}\] 
	is dense in $X$. 
\end{lemma}

\begin{proof} Let $x_0 \in S_X$. Pick $x^* \in S_{X^*}$ to be such that $x^*(x_0) = 1$. Now, let us consider the closed convex set $C:= \{x\in B_{X} : x^*(x)=1\}$. Since $X$ has the RNP, we have that $C=\cconv\ext{C}$. Moreover, $C$ is a face of $B_{X}$ and so $\ext{C}\subseteq \ext{B_X}$. Thus, $x_0 \in \cconv\{x\in \ext{B_X} : x^*(x)=1\}$. To conclude, it suffices to check that 
\begin{equation*} 
\co\{x\in \ext{B_X}: x^*(x)=1\}\subseteq A.
\end{equation*} 
To this end, take $v=\sum_{i=1}^n \lambda_i x_i$, where $x_i\in \ext{B_X}$, $x^*(x_i)=1$ and $\lambda_i>0$ for all $i=1, \ldots, n$, and $\sum_{i=1}^n \lambda_i =1$. Then,
	\[ 1\geq \norm{v} \geq \langle x^*, v\rangle = \sum_{i=1}^n \lambda_i =1 \]
	and so $v\in A$. A straightforward homogeneity argument allows us to restrict the assumption $\Vert x_0\Vert=1$, and the lemma is proved. \end{proof}

We also need the following result, which is a consequence of Lemma~\ref{characterization} and Lemma~\ref{lemma:RNP}.

\begin{lemma}\label{th:RNPsym} Let $X$ be a Banach space. Assume that $\sten X$ has the RNP and that $\ext{B_{\sten X}}\subseteq \{\pm x^N: x\in B_X \}$. Then, every symmetric tensor can be approximated by symmetric tensors which attain their norms. In other words, 
\begin{equation*} 
\overline{\NA_{\pi, s, N} (X)}^{\|\cdot\|_{\pi,s,N}} = \sten X. 
\end{equation*} 
\end{lemma}	
\begin{proof}
	By Lemma~\ref{lemma:RNP}, the set 
	\begin{align*} A&= \left\{ z=\sum_{i=1}^n \varepsilon_i\lambda_i x_i^N\in \sten X: \varepsilon_i\in \{1,-1\}, \lambda_i>0, x_i\in S_X,  \norm{z}=\sum_{i=1}^n \lambda_i \right\}\\
		&= \left\{ z=\sum_{i=1}^n \lambda_i x_i^N\in \sten X: \lambda_i\in \mathbb R, x_i\in S_X, \norm{z}=\sum_{i=1}^n |\lambda_i| \right\}
		\end{align*}
	is dense in $X$. Clearly, $A\subseteq \NA_{\pi,s,N}(X)$. 
\end{proof}

Now we are ready to prove Theorem~\ref{RNP+AP=symmetric}.

\begin{proof}[Proof of Theorem~\ref{RNP+AP=symmetric}]  Let us observe first that if $X^*$ has the RNP and the AP, then $\sten X^*$ has the RNP. Indeed, by using \cite[Subsection 2.3]{F1}, we have that $\sten X^*$ is isomorphic to a subspace of $X^* \pten \dots \pten X^*$. Since $X^*$ has the RNP and AP, we have that $X^* \pten \dots \pten X^*$ has the RNP (see \cite[Theorem~VIII.4.7]{DU}) and then we can conclude that $\sten X^*$ has the RNP. Now by using \cite[Proposition~1]{BR}, we have that
	\begin{equation*}
		\ext{B_{\sten X^*}} = \ext{B_{(\siten X)^*}} \subseteq \{\pm \varphi^N: \varphi \in X^*, \|\varphi\| = 1\} 
	\end{equation*}
	and by Lemma~\ref{th:RNPsym}, the set $\NA_{\pi, s, N} (X^*)$ is dense in $\sten X^*$, as desired. 
\end{proof}

Notice that if $X^*$ has the RNP, then $\mathcal{P}_I (^N X)$, the Banach space of all $N$-homogeneous integral polynomials on $X$, has the RNP. Consequently, $\siten X$ cannot contain an isomorphic copy of $\ell_1$, which in turn implies that $\mathcal{P}_I (^N X)$ is isometrically isomorphic to the Banach space $\mathcal{P}_{\text{nu}} (^N X)$ of all $N$-homogeneous nuclear polynomials on $X$ (see \cite[Theorem~2]{BR}).

\begin{theorem}\label{thm:no_ell_1}
Let $X$ be a Banach space. Suppose that $\siten X$ does not contain a copy of $\ell_1$. Then the set of norm-attaining elements in $\mathcal{P}_{\text{nu}} (^N X)$ is $w^*$-dense in $\mathcal{P}_{\text{nu}} (^N X) = (\siten X)^*$. 
\end{theorem}

\begin{proof}
Let $P \in S_{\mathcal{P}_{\text{nu}} (^N X)}$ be given. By the Bishop-Phelps theorem \cite{BP}, given $\eps >0$, we can find $P_0 \in S_{(\siten X)^*}$ so that $P_0$ attains its norm at some $u_0 \in S_{\siten X}$ and $\| P_0 - P \| < \eps$. We will prove that $P_0$ can be approximated by norm-attaining elements in $\mathcal{P}_{\text{nu}} (^N X)$ in the $w^*$-topology. For this, let us consider the set 
\begin{equation*} 
C := \Big\{ Q \in B_{(\siten X)^*} : \langle Q, u_0 \rangle =1 \Big\}.
\end{equation*} 
Notice that $C$ is a $w^*$-compact and convex set. It follows from Krein-Milman theorem (see, for instance, \cite[Theorem 7.68]{AB}) that $C = \overline{\text{co}}^{w^*} ( \ext{C} )$. As $C$ is a face of $B_{(\siten X)^*}$, thanks to \cite[Proposition~1]{BR}, we have that 
\begin{equation*} 
C \subseteq \overline{\text{co}}^{w^*} \Big( \Big\{ \pm (x^*)^N  : x^* \in S_{X^*}, \, \langle \pm (x^*)^N , u_0 \rangle =1 \Big\} \Big). 
\end{equation*} 
It follows from Lemma~\ref{characterization} that $P_0$ can be approximated by norm-attaining elements in $\mathcal{P}_{\text{nu}} (^N X)$ in the $w^*$-topology and we are done.
\end{proof} 

Recall that $\mathcal{P}_{\text{nu}} (^N X)$ coincides with $\sten X^*$ isometrically whenever $X^*$ has the AP. It is known that the James-Hagler space $JH$ is an example of a Banach space whose dual does not have the RNP while the symmetric injective tensor product $\siten JH$ does not contain a copy of $\ell_1$ (see \cite{Hag}).  
Thus, the assumption in Corollary~\ref{cor:no_ell_1} below is strictly weaker than that of Theorem~\ref{RNP+AP=symmetric}. 

\begin{corollary}\label{cor:no_ell_1}
Let $X$ be a Banach space such that $X^*$ has the AP. If $\siten X$ does not contain a copy of $\ell_1$, then the set $\NA_{\pi, s, N} (X^*)$ is $w^*$-dense in $\sten X^*$. 
\end{corollary}

Let us observe that so far we have presented only positive results on the ($w^*$-) denseness of symmetric tensors which attain their norms in $\sten X$. In fact, we do not know whether the set $\NA_{\pi,s,N}(X)$ is dense in $\sten X$ for {\it every} Banach space $X$. The first candidate that would pop up in our minds would be a Banach space $X$ such that the set $\NA_{\pi}(X \pten X)$ is not dense in $X \pten X$. Nevertheless, the techniques from \cite[Section 5]{DJRR} (where the authors show that there exist subspaces $X$ of $c_0$ and $Y$ of the Read's space $\mathcal{R}$ such that the set $\NA_{\pi}(X \pten Y^*)$ is not dense) do not seem to work. Indeed, the idea behind was requiring that every element of $\NA(X,Y^*)$ has finite rank, and then working with a bounded operator $T:X\longrightarrow Y$ which can not be approximated by finite rank operators from the failure of the approximation property. This construction is doable since $X$ and $Y$ are isomorphic and then $T$ can be taken as a formal identity thanks to classical results on AP. However, for one such example of the form $X\pten X$, we would need to work with an operator $T:X\longrightarrow X^*$, for a certain subspace $X$ of $c_0$, which is not approximable by finite rank operators and, to the best of our knowledge, the existence of such a space $X$ and such a $T$ is unknown. Despite that, we shall conclude this section by showing that this open problem is separably determined. 

\begin{theorem}\label{thm:separable_determined}
	Let $N \in \N$ be fixed. If $\NA_{\pi, s, N} (Y)$ is dense in $\sten Y$ for every separable Banach space $Y$, then $\NA_{\pi, s, N} (X)$ is dense in $\sten X$ for every Banach space $X$. 
\end{theorem}

\begin{proof} 
	Let $X$ be a Banach space, $z \in \sten X$, and let $\eps >0$ be given. Choose a representation $z = \sum_{n=1}^\infty \lambda_n x_n^N$ with $(\lambda_n)_{n=1}^{\infty} \subseteq \mathbb{R} \setminus \{0\}$ and $(x_n)_{n=1}^{\infty} \subseteq S_X$ satisfying that $\sum_{n=1}^\infty |\lambda_n| < \| z \|_{\sten X} + \eps$. Let $Z := \overline{\spann} \{ x_n : n \in \N \}$. Thus, $Z$ is a separable Banach space. By \cite[Proposition~2]{Sims-Yost} (see also  \cite[Lemma~4.3]{HWW}), there exists a separable ideal $Y$ of $X$ such that $Z \subseteq Y$. As $Y$ is an ideal, by Theorem~\ref{ideal}, we have that $\| z \|_{\sten X} = \| z \|_{\sten Y}$. By the hypothesis, there exists $z' = \sum_{n=1}^\infty \mu_n y_n^N \in \NA_{\pi, s, N} (Y)$ with $\|z'\|_{\sten Y} = \sum_{n=1}^\infty |\mu_n|$ which satisfies that $\| z- z'\|_{\sten Y} < \eps$. Considering $z'$ as an element of $\sten X$, we notice that 
	\[
	\sum_{n=1}^\infty |\mu_n| = \|z'\|_{\sten Y} = \|z' \|_{\sten X} \leq \sum_{n=1}^\infty |\mu_n|,
	\]
	which implies that $z' \in \NA_{\pi, s, N} (X)$. Finally, $\|z-z'\|_{\sten X} = \| z-z'\|_{\sten Y} < \eps$. 
\end{proof}

\section{Results for Projective Tensor Products}

In this section, we present some results on the denseness of tensors in projective tensor products of Banach spaces.

Let us first notice that when $X = L_1(\mathbb{T})$, where the unit circle $\mathbb{T}$ is equipped with the Haar measure, and $Y$ is the two-dimensional Hilbert space $\ell_2^2$, we have that $\NA_{\pi}(X \pten Y) \not= X \pten Y$ \cite[Example 3.12 (a)]{DJRR}. This shows that finite dimensionality on just one of the factors is not enough to guarantee that every tensor in $X \pten Y$ is norm-attaining. Nevertheless, we have the following result.

\begin{theorem} \label{theorem-polyhedral-dual} Let $X$ be a Banach space with $B_X = \co{(\{x_1, \ldots, x_n\})}$ for some $x_1, \ldots, x_n \in S_X$ and assume that $Y$ is a dual space. Then, every tensor in $X \pten Y$ attains its projective tensor norm. In other words, 
\begin{equation*} 
\NA_{\pi} (X \pten Y) = X \pten Y.
\end{equation*} 
\end{theorem}

\begin{proof} Let us assume that $Y = Z^*$. Notice first that since $X$ is finite dimensional, we have that $\mathcal{L}(X, Z) = \mathcal{K}(X, Z) = X^* \iten Z$ and then $\mathcal{L}(X, Z)^* = X \pten Z^*$. We will use this fact to prove the following statement. 
	
	\vspace{0.2cm}	
	\noindent	
	{\it Claim}: The set $\{x_i\} \otimes B_{Z^*}$ is a $w^*$-compact convex subset of $X \pten Z^*$. 
	\vspace{0.2cm}

	Indeed, for each $i=1,\ldots,n$, let us take  $T_i\colon \mathcal{L}(X, Z) \longrightarrow Z$ to be defined by $T_i(T) := T(x_i)$ for every $T \in \mathcal{L}(X, Z)$. Therefore, its adjoint operator $T_i^*\colon Z^* \longrightarrow \mathcal{L}(X, Z)^* = X \pten Z^*$ satisfies $T_i^*(z^*) = x_i \otimes z^*$ for every $z^* \in Z^*$ and $i=1, \ldots, n$. This implies that $T_i^*(B_{Z^*}) = x_i \otimes B_{Z^*}$ and since $T^*$ is $w^*$-$w^*$ continuous, we can conclude that $\{x_i\} \otimes B_{Z^*}$ is $w^*$-compact convex in $X \pten Z^*$. 
	
	\vspace{0.2cm}	
	Thus, $A := \co{ \left( \bigcup_{i=1}^n x_i \otimes B_{Z^*} \right)}$ is $w^*$-compact as being the convex hull of a finite number of $w^*$-compact convex sets (see, for instance, \cite[Lemma~5.29]{AB}). So, in particular, $A$ is $w^*$-closed and then norm-closed. Finally, if $z \in B_X \otimes B_{Z^*}$, then there are $\lambda_i>0$ with $\sum_{i=1}^n \lambda_i =1$ such that
	\begin{equation*}
	z = \left( \sum_{i=1}^n \lambda_i x_i \right) \otimes y = \sum_{i=1}^n \lambda_i x_i \otimes y.
	\end{equation*}
Therefore, $B_X \otimes B_{Z^*} \subseteq A$ and by taking convex hulls, we get $A = B_{X \pten Z^*} = B_{X \pten Y}$. In particular, every element of $S_{X\pten Z^*}$ can be written as a finite convex combination of basic tensors in $B_X\otimes B_{Z^*}$, so it is norm attaining. 
\end{proof}

Recall a Banach space $X$ is said to be polyhedral if the unit ball of every finite-dimensional subspace is a polytope, that is, the convex hull of a finite set. We can use Theorem~\ref{theorem-polyhedral-dual} to get the following denseness result.

\begin{theorem} \label{theorem-polyhedral} Let $X$ be a Banach which is polyhedral and satisfies the metric $\pi$-property. Assume that $Y$ is a dual space. Then, every tensor in $X \pten Y$ can be approximated by tensors that attain their norms. In other words,
\begin{equation*} 
\overline{\NA_{\pi} (X \pten Y)}^{\|\cdot\|_{\pi}} = X \pten Y. 
\end{equation*} 
\end{theorem}

\begin{proof} Let $u \in S_{X \pten Y}$ and $\eps \in (0, 1)$ be given. Then, there exist sequences $(\lambda_n) \subseteq \R^+$, $(x_n) \subseteq S_X$, and $(y_n) \subseteq S_Y$ with $u = \sum_{n=1}^{\infty} \lambda_n x_n \otimes y_n$ and $\sum_{n=1}^{\infty} \lambda_n < 1 + \eps$. We may find $k \in \N$ so that $\|u - z\| < \frac{\eps}{2}$, where $z := \sum_{n=1}^k \lambda_n x_n \otimes y_n$. Since $X$ satisfies the metric $\pi$-property, we can find a finite-dimensional subspace $M$ of $X$ which is 1-complemented and such that for every $n \in \{1, \ldots, k\}$, there exists $x_n' \in M$ such that $\|x_n - x_n'\| < \frac{\eps}{4}$. Define $z' := \sum_{n=1}^k \lambda_n x_n' \otimes y_n$ and notice that
	\begin{eqnarray*}
		\|z' - z\| = \left\| \sum_{n=1}^{k} \lambda_n x_n' \otimes y_n - \sum_{n=1}^k \lambda_n x_n \otimes y_n \right\| &=& \left\| \sum_{n=1}^k \lambda_n (x_n' - x_n) \otimes y_n \right\| \\
		&\leq& \sum_{n=1}^{\infty} \lambda_n \|x_n' - x_n\| \|y_n\| \\
		&<& \frac{\eps}{4} \sum_{n=1}^k \lambda_n \\
		&<& \frac{\eps}{2}. 
	\end{eqnarray*}
	Notice now that $z' \in M \pten Y$ and since $M$ is 1-complemented in $X$, we have that $\|z'\|_{M \pten Y} = \|z'\|_{X \pten Y}$. Moreover, since $X$ is a polyhedral, we have that $B_M$ is equal to $\co{\{x_1, \ldots, x_m\}}$ for some $x_1, \ldots, x_m \in S_M$. Theorem~\ref{theorem-polyhedral-dual} shows then that $z' \in \NA_{\pi} (M \pten Y)$ and, therefore, $z' \in \NA_{\pi}(X \pten Y)$. Finally, $\|z' - u\| \leq \|z' - z\| + \|z - u\| < \eps$.	
\end{proof}

The Banach space $c_0$ endowed with $\|\cdot\|_{\infty}$ is the canonical example of a polyhedral space. So, we have the following immediate consequence of Theorem~\ref{theorem-polyhedral}, which is not covered by \cite[Theorem~4.8]{DJRR}.

\begin{corollary} Let $Y$ be a dual space. Then, $\NA_{\pi}(c_0 \pten Y)$ is dense in $c_0 \pten Y$.
\end{corollary}


\begin{remark} Notice that in Theorem~\ref{theorem-polyhedral} the hypothesis of $X$ having the metric $\pi$-property is essential. Indeed, in \cite[Section~5]{DJRR} the authors show that if $X$ is a closed subspace of $c_0$ (and hence polyhedral since this property is hereditary) failing the approximation property and $Y:= (X, \nn{\cdot})$ is a renorming of $X$, where $\nn{\cdot}$ is the norm that defines Read's space, then $\NA_{\pi} (X \pten Y^*)$ is not dense in $X \pten Y^*$.
\end{remark}

Next, we can prove the following result on the denseness of nuclear operators which attain their nuclear norms under the RNP assumption. 

\begin{theorem}\label{th:RNPnuclear} Let $X, Y$ be Banach spaces such that $X^*$ and $Y^*$ have the RNP. Then, every nuclear operator from $X$ into $Y^*$ can be approximated by norm-attaining nuclear operators. In other words, 
\begin{equation*} 
\overline{\NA_{\mathcal N}(X, Y^*)}^{\|\cdot\|_{\mathcal{N}}} = \mathcal{N}(X, Y^*).
\end{equation*} 
\end{theorem}

\begin{proof} Suppose that $X^*$ and $Y^*$ have the RNP. Then, $\mathcal N(X, Y^*)=(X\iten Y)^*$ also has the RNP (this is shown in \cite[Theorem~VIII.4.7, pg. 249]{DU} under the additional assumption that $X^*$ or $Y^*$ have the AP, which is only used to get $\mathcal N(X, Y^*)=X^*\pten Y^*$). Also, $\ext{B_{\mathcal N(X,Y^*)}}=\ext{B_{(X\iten Y)^*}}\subseteq S_{X^*}\otimes S_{Y^*}$ (cf. \cite{RS}). By Lemma~\ref{lemma:RNP}, the set 
	\[ A= \left\{T=\sum_{i=1}^n\lambda_i x_i^*\otimes y_i^*: \lambda_i>0, x_i^*\in S_{X^*}, y_i^*\in S_{Y^*} \text{ for } i=1, \ldots, n,   \norm{T}=\sum_{i=1}^n \lambda_i \right\}\]
is dense in $\mathcal N(X, Y^*)$. Clearly, $A \subseteq  \NA_{\mathcal{N}}(X, Y^*)$. 
\end{proof}

If we are under the hypotheses of Theorem~\ref{th:RNPnuclear} together with the extra assumption that one of the spaces has the approximation property, then the equality $X^*\pten Y^*=\mathcal N(X,Y^*)$ holds (see, for instance, \cite[Corollary~4.8]{R}). By using Theorem~\ref{th:RNPnuclear}, we get the following counterpart of Theorem~\ref{RNP+AP=symmetric} for non-symmetric tensors, which provides a positive answer for \cite[Question~6.1]{DJRR} in the case that one of the spaces has the approximation property.

\begin{corollary}\label{cor:RNPtensor} Let $X, Y$ be Banach spaces such that $X^*$ and $Y^*$ have the RNP and at least one of them has the AP. Then, every tensor in $X^* \pten Y^*$ can be approximated by tensors that attain their norms. In other words, 
\begin{equation*} 	
\overline{\NA_\pi(X^*\pten Y^*)}^{\|\cdot\|_{\pi}} = X^*\pten Y^*.
\end{equation*} 
\end{corollary}

Even if we weaken the assumption in Theorem~\ref{th:RNPnuclear} so that only $X^*$ has the RNP, we are still able to obtain the denseness result, but in the $w^*$-topology, of the set of norm-attaining nuclear operators as in Theorem~\ref{thm:no_ell_1}. Notice that $\mathcal{N}(X,Y^*)$ is identified with $(X \iten Y)^*$, the dual of injective tensor space, under the assumption that $X^*$ has the RNP. Arguing in the same way as in the proof of Theorem~\ref{thm:no_ell_1} but using the fact that $\ext{B_{(X\iten Y)^*}}\subseteq S_{X^*}\otimes S_{Y^*}$ and \cite[Theorem~3.2]{DJRR} instead of \cite[Proposition~1]{BR} and Lemma~\ref{characterization}, respectively, the following result can be obtained.

\begin{theorem} \label{nuclear-weak-star} Let $X$ be a Banach space such that $X^*$ has the RNP. Then, the set $\NA_{\mathcal{N}} (X, Y^*)$ is $w^*$-dense in $\mathcal{N} (X, Y^*) = (X \iten Y)^*$ for any Banach space $Y$.  
\end{theorem}

Using the equality $X^*\pten Y^*=\mathcal N(X,Y^*)$ provided that one of $X^*$ or $Y^*$ has the AP, we get the following immediate consequence of Theorem~\ref{nuclear-weak-star}.

\begin{corollary} Let $X , Y$ be Banach spaces such that $X^*$ has the RNP and at least one of $X^*$ or $Y^*$ has the AP. Then, the set $\NA_{\pi} (X^* \pten Y^*)$ is $w^*$-dense in $X^* \pten Y^*=(X\iten Y)^*$. 
\end{corollary}

\noindent
\textbf{Acknowledgements:} The authors would like to thank Daniel Carando and Jorge Tom\'as Rodr\'iguez for fruitful conversations on the topic of the paper. The first author was supported by the project OPVVV CAAS CZ.02.1.01/0.0/0.0/16\_019/0000778 and by the Estonian Research Council grant PRG877. The second author is supported in part by the grants MTM2017-83262-C2-2-P and Fundaci\'on S\'eneca Regi\'on de Murcia 20906/PI/18. The third author was supported by NRF (NRF- 2019R1A2C1003857). The fourth author was supported by Juan de la Cierva-Formaci\'on fellowship FJC2019-039973, by MTM2017-86182-P (Government of Spain, AEI/FEDER, EU), by MICINN (Spain) Grant PGC2018-093794-B-I00 (MCIU, AEI, FEDER, UE), by Fundaci\'on S\'eneca, ACyT Regi\'on de Murcia grant 20797/PI/18, by Junta de Andaluc\'ia Grant A-FQM-484-UGR18 and by Junta de Andaluc\'ia Grant FQM-0185.

\end{document}